\documentclass[submission,copyright,creativecommons]{eptcs}
\providecommand{\event}{QPL 2015} 
\usepackage{breakurl}             

\usepackage{amsmath,amssymb,amsthm}
\usepackage{enumitem}
\setlist{nosep}

\usepackage{tikz}
\usetikzlibrary{arrows,positioning,fit,matrix,shapes.geometric,matrix,intersections,decorations.markings}

\newcommand*\pgfdeclareanchoralias[3]{%
  \expandafter\def\csname pgf@anchor@#1@#3\expandafter\endcsname
     \expandafter{\csname pgf@anchor@#1@#2\endcsname}}

\tikzset{
circnode/.style={
  circle, draw=red, very thin, outer sep=0.025em, minimum size=2em,
  fill=red, text centered},
integral/.style={
  regular polygon, regular polygon sides=3, shape border rotate=180, draw=black, very thick,
  outer sep=0.025em, inner sep=0, minimum size=2em, fill=blue!5, text centered},
multiply/.style={
  regular polygon, regular polygon sides=3, shape border rotate=180, draw=black, very thick,
  outer sep=0.025em, inner sep=0, minimum size=2em, fill=blue!5, text centered},
upmultiply/.style={
  regular polygon, regular polygon sides=3, draw=black, very thick,
  outer sep=0.025em, inner sep=0, minimum size=2em, fill=blue!5, text centered},
zero/.style={
  circle, draw=black, very thick, minimum size=0.15cm, fill=black,
  inner sep=0, outer sep=0},
hole/.style={
  circle, draw=white, very thick, minimum size=0.25cm, fill=white,
  inner sep=0, outer sep=0},
bang/.style={
  circle, draw=black, very thick, minimum size=0.15cm, fill=green!10,
  inner sep=0, outer sep=0},
delta/.style={
  regular polygon, regular polygon sides=3, minimum size=0.4cm, inner
  sep=0, outer sep=0.025em, draw=black, very thick, fill=green!10},
codelta/.style={
  regular polygon, regular polygon sides=3, shape border rotate=180, minimum size=0.4cm,
  inner sep=0, outer sep=0.025em, draw=black, very thick, fill=green!10},
plus/.style={
  regular polygon, regular polygon sides=3, shape border rotate=180, minimum size=0.4cm,
  inner sep = 0, outer sep=0.025em, draw=black, very thick, fill=black},
coplus/.style={
  regular polygon, regular polygon sides=3, minimum size=0.4cm,
  inner sep = 0, outer sep=0.025em, draw=black, very thick, fill=black},
sqnode/.style={
  regular polygon, regular polygon sides=4, minimum size=2.6em,
  draw=black, very thick, inner sep=0.2em, outer sep=0.025em,
  fill=yellow!10, text centered},
blackbox/.style={
  regular polygon, regular polygon sides=4, minimum size=2.6em,
  draw=black, very thick, inner sep=0.2em, outer sep=0.025em, fill=black},
bigcirc/.style={
  circle, draw=black, very thick, text width=1.6em, outer sep=0.025em,
  minimum height=1.6em, fill=blue!5, text centered}
 }

\pgfdeclareanchoralias{regular polygon}{corner 1}{io}
\pgfdeclareanchoralias{regular polygon}{corner 2}{left out}
\pgfdeclareanchoralias{regular polygon}{corner 3}{right out}
\pgfdeclareanchoralias{regular polygon}{corner 3}{left in}
\pgfdeclareanchoralias{regular polygon}{corner 2}{right in}

\tikzset{
    overdraw/.style={preaction={draw,white,line width=#1}},
    overdraw/.default=5pt
}

\usepackage{color}

\definecolor{myurlcolor}{rgb}{0.6,0,0}
\definecolor{mycitecolor}{rgb}{0,0,0.8}
\definecolor{myrefcolor}{rgb}{0,0,0.8}
\hypersetup{colorlinks, linkcolor=myrefcolor, citecolor=mycitecolor, urlcolor=myurlcolor}

\newcommand{\End}{\mathrm{End}}
\newcommand{\Hom}{\mathrm{Hom}}
\newcommand{\Mat}{\mathrm{Mat}}
\newcommand{\MR}{\Mat(R)}
\newcommand{\id}{\mathrm{id}}

\newcommand{\Cc}{\mathcal C}
\newcommand{\Fb}{\mathbf F}
\newcommand{\Gb}{\mathbf G}

\newcommand{\FinRel}{\mathrm{FinRel}}
\newcommand{\FinSpan}{\mathrm{FinSpan}}
\newcommand{\FinVect}{\mathrm{FinVect}}
\newcommand{\Bimon}{\mathrm{Bimon}}

\newcommand{\R}{{\mathbb R}}
\newcommand{\C}{{\mathbb C}}

\newcommand{\N}{{\mathbb N}}

\newcommand{\Z}{{\mathbb Z}}
\newcommand{\B}{{\mathbb B}}

\newtheorem{thm}{Theorem}
\newtheorem*{thm*}{Theorem}

\newtheorem{lemma}[thm]{Lemma}
\newtheorem{cor}[thm]{Corollary}

\title{PROPs for Linear Systems}
\author{
Simon Wadsley
\institute{Homerton College\\Cambridge\\CB2 8PH\\UK}
\email{S.J.Wadsley@dpmms.cam.ac.uk}
\and
Nick Woods
\institute{Department of Mathematics\\University of California\\Riverside, CA 92521\\USA}
\email{woods@math.ucr.edu}
}
\def\titlerunning{PROPs for Linear Systems}
\def\authorrunning{Simon Wadsley \& Nick Woods}
\begin{document}
\maketitle

\begin{abstract}
A PROP is a symmetric monoidal category whose objects are the nonnegative integers and whose tensor product on objects is addition. A morphism from $m$ to $n$ in a PROP can be visualized as a string diagram with $m$ input wires and $n$ output wires. For a field $k$, the PROP $\FinVect_k$ where morphisms are $k$-linear maps is used by Baez and Erbele to study signal-flow diagrams. We aim to generalize their result characterizing this PROP in terms of generators and relations by looking at the PROP $\Mat(R)$ of matrices with values in $R$, where $R$ is a commutative rig (that is, a generalization of a ring where the condition that each element has an additive inverse is relaxed). To this end, we show that the category of symmetric monoidal functors out of $\Mat(R)$ is equivalent to the category of bicommutative bimonoids equipped with a certain map of rigs; such functors are called algebras. By choosing $R$ correctly, we will see that the algebras of the PROP $\FinSpan$ of finite sets and spans between them are bicommutative bimonoids, while the algebras of the PROP $\FinRel$ of finite sets and relations between them are special bicommuative bimonoids and the algebras of $\Mat(\Z)$ are bicommutative Hopf monoids.
\end{abstract}

\section{Introduction}

{\it Product and permutation categories}, or PROPs for short, are tools used to describe the algebraic structure of an object. They were introduced by Mac Lane for the purposes of universal algebra, generalizing Lawvere's notion of an algebraic theory \cite{MLca,Mar}. 

Formally, a PROP is a symmetric monoidal functor whose objects are the natural numbers and whose tensor product on objects is given by ordinary addition. If $\mathcal T$ is a PROP and $\Cc$ is a symmetric monoidal category, we define an algebra of $\mathcal T$ over $\Cc$ to be a symmetric monoidal functor $\mathcal T \rightarrow \Cc$. We say that $\mathcal T$ is the PROP for a category $\mathcal A$ if the category of algebras over $\mathcal T$ is equivalent to $\mathcal A$. 

PROPs are useful in providing a mathematical formalization of various sorts of diagrams used in physics and engineering, such as signal-flow diagrams \cite{Erb,BSZ2}. For example, a morphism $m \rightarrow n$ can be visualized as a black box diagram with $m$ inputs and $n$ outputs.
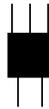
\begin{center}
\scalebox{0.80}{
\begin{tikzpicture}[thick]
\node [blackbox] (b1) at (0,0) {$f$};
\node (b) at (0,1) {};
\node (a1) at (-0.3,0) {};
\node (a) at (-0.3,1) {};
\node (c1) at (0.3,0) {};
\node (c) at (0.3,1) {};
\node (d1) at (-0.2,0) {};
\node (d) at (-0.2,-1) {};
\node (e1) at (0.2,0) {};
\node (e) at (0.2,-1) {};

\draw (b1) -- (b)
      (a1) -- (a)
      (c1) -- (c)
      (d1) -- (d)
      (e1) -- (e);
\end{tikzpicture}
}
\end{center}
The monoidal category structure provides us with different ways of sticking components together --- composing two morphisms $f: m \rightarrow n$ and $g: n \rightarrow p$ lets us connect the outputs of $f$ to the inputs of $g$, while tensoring provides a way of placing two circuits side-by-side.
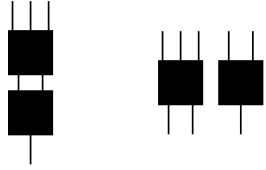
\begin{center}
\scalebox{0.80}{
\begin{tikzpicture}[thick]
\node [blackbox] (b1) at (0,0) {$f$};
\node (b) at (0,1) {};
\node (a1) at (-0.3,0) {};
\node (a) at (-0.3,1) {};
\node (c1) at (0.3,0) {};
\node (c) at (0.3,1) {};
\node (d1) at (-0.2,0) {};
\node (d) at (-0.2,-1) {};
\node (e1) at (0.2,0) {};
\node (e) at (0.2,-1) {};

\node [blackbox] (f1) at (0,-1) {};
\node (f) at (0,-2) {};

\draw (b1) -- (b)
      (a1) -- (a)
      (c1) -- (c)
      (d1) -- (d)
      (e1) -- (e)
      (f1) -- (f);
\end{tikzpicture}
\qquad \qquad
\begin{tikzpicture}[thick]
\node (filler) at (0,-1.5) {};

\node [blackbox] (b1) at (0,0) {$f$};
\node (b) at (0,1) {};
\node (a1) at (-0.3,0) {};
\node (a) at (-0.3,1) {};
\node (c1) at (0.3,0) {};
\node (c) at (0.3,1) {};
\node (d1) at (-0.2,0) {};
\node (d) at (-0.2,-1) {};
\node (e1) at (0.2,0) {};
\node (e) at (0.2,-1) {};

\node [blackbox] (h1) at (1,0) {};
\node (h) at (1,-1) {};
\node (f1) at (0.8,0) {};
\node (f) at (0.8,1) {};
\node (g1) at (1.2,0) {};
\node (g) at (1.2,1) {};

\draw (b1) -- (b)
      (a1) -- (a)
      (c1) -- (c)
      (d1) -- (d)
      (e1) -- (e)
      (f1) -- (f)
      (g1) -- (g)
      (h1) -- (h);
\end{tikzpicture}
}
\end{center}

A rig is generalization of a ring (with unit) where elements are not required to have additive inverses. Some common examples include the natural numbers $\N$ and the two-element Boolean rig $\B$ where 0 and 1 are the truth values FALSE and TRUE, respectively, and where addition is given by OR and multiplication by AND. For a commutative rig $R$, we define a PROP $\Mat(R)$ by letting morphisms $m \rightarrow n$ be $n \times m$ matrices with values in $R$. Note that $\Mat(R)$ includes some famous categories as particular cases; for example, $\Mat(\B)$ is equivalent to the category $\FinRel$ of finite sets and relations between them, and $\Mat(\N)$ is equivalent to the category $\FinSpan$ of finite sets and spans between them. If $k$ is a field, $\Mat(k)$ is equivalent to the category of $\FinVect_k$ of finite $k$-vector spaces and linear maps between them. 

Note that $\Mat(R)$ generalizes Heisenberg's matrix mechanics; $R = \C$ gives the familiar case. If we let $R = [0, \infty)$, then $\Mat(R)$ gives an analogue of matrix mechanics where relative probabilities replace amplitudes. If $R$ is the Boolean rig $\B$, relative possibilities replace amplitudes instead.

We will show:

\begin{thm*}
Let $R$ be a commutative rig. Then $\Mat(R)$ is the PROP for bicommutative bimonoids $A$ equipped with a rig map from $R$ to the rig of bimonoid endomorphisms on $A$.
\end{thm*}

In the case that $k$ is the field $\R(s)$ of rational functions in one variable, the string diagrams of the PROP $\Mat(k)$ are the signal-flow diagrams studied by Baez and Erbele and Soboci\'nski et al.\ \cite{Erb,BSZ1,BSZ2}. Baez and Erbele show that, when coming up with a set of generators and relations of the category $\Mat(k)$ where $k$ is a field, the additive and multiplicative inverses do not play a role \cite{Erb}. Instead of describing the PROP in terms of generators and relations, our focus on describing functors out of $\Mat(R)$ give a similar characterization, but the proof is much more efficient. Our method also allows $\Mat(R)$ to be compared to PROPs describing similar structures, such as the one for bimonoids without unit or counit given in \cite{Pir}. Some consequences of our theorem are that
\begin{itemize}
  \item $\FinSpan$ is the PROP for bicommutative bimonoids;
  \item $\FinRel$ is the PROP for special bicommutative bimonoids; and
  \item $\Mat(\Z)$ is the PROP for bicommutative Hopf monoids.
\end{itemize}

\section{Bicommutative Bimonoids} \label{bimon}

Fix a monoidal category $\Cc$ with symmetry functor $\tau_{AB}$. A {\it commutative monoid} is an object $A \in \Cc$ with functors $\mu_A: A \otimes A \rightarrow A$ and $\eta_A: I \rightarrow A$, called multiplication and the unit respectively, satisfying associativity, the unit laws, and commutativity. Diagrammatically, we represent the multiplication and unit operations as 

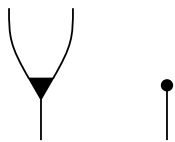
\begin{center}
\scalebox{0.85}{
   \begin{tikzpicture}[thick]
   \node[plus] (adder) {};
   \node (f) at (-0.5,1.35) {};
   \node (g) at (0.5,1.35) {};
   \node (out) [below of=adder] {};

   \draw (f) .. controls +(-90:0.6) and +(120:0.6) .. (adder.left in);
   \draw (g) .. controls +(-90:0.6) and +(60:0.6) .. (adder.right in);
   \draw (adder) -- (out);
   \end{tikzpicture}
}
\qquad
\scalebox{0.85}{
  \begin{tikzpicture}[thick]
   \node (out1) {};
   \node [zero] (ins1) at (0,1) {};

   \draw (out1) -- (ins1);
   \end{tikzpicture}
}
\end{center}
and the relations they must satisfy are depicted as
 \begin{center}
    \scalebox{0.80}
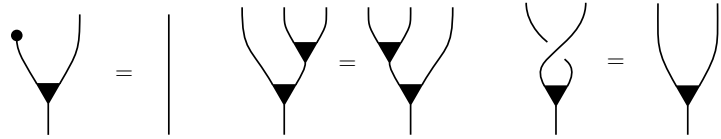
{
   \begin{tikzpicture}[-, thick, node distance=0.74cm]
   \node [plus] (summer) {};
   \node [coordinate] (sum) [below of=summer] {};
   \node [coordinate] (Lsum) [above left of=summer] {};
   \node [zero] (insert) [above of=Lsum, shift={(0,-0.35)}] {};
   \node [coordinate] (Rsum) [above right of=summer] {};
   \node [coordinate] (sumin) [above of=Rsum] {};
   \node (equal) [right of=Rsum, shift={(0,-0.26)}] {\(=\)};
   \node [coordinate] (in) [right of=equal, shift={(0,1)}] {};
   \node [coordinate] (out) [right of=equal, shift={(0,-1)}] {};

   \draw (insert) .. controls +(270:0.3) and +(120:0.3) .. (summer.left in)
         (summer.right in) .. controls +(60:0.6) and +(270:0.6) .. (sumin)
         (summer) -- (sum)    (in) -- (out);
   \end{tikzpicture}
        \hspace{1.0cm}
   \begin{tikzpicture}[-, thick, node distance=0.7cm]
   \node [plus] (uradder) {};
   \node [plus] (adder) [below of=uradder, shift={(-0.35,0)}] {};
   \node [coordinate] (urm) [above of=uradder, shift={(-0.35,0)}] {};
   \node [coordinate] (urr) [above of=uradder, shift={(0.35,0)}] {};
   \node [coordinate] (left) [left of=urm] {};

   \draw (adder.right in) .. controls +(60:0.2) and +(270:0.1) .. (uradder.io)
         (uradder.right in) .. controls +(60:0.35) and +(270:0.3) .. (urr)
         (uradder.left in) .. controls +(120:0.35) and +(270:0.3) .. (urm)
         (adder.left in) .. controls +(120:0.75) and +(270:0.75) .. (left)
         (adder.io) -- +(270:0.5);

   \node (eq) [right of=uradder, shift={(0,-0.25)}] {\(=\)};

   \node [plus] (ulsummer) [right of=eq, shift={(0,0.25)}] {};
   \node [plus] (summer) [below of=ulsummer, shift={(0.35,0)}] {};
   \node [coordinate] (ulm) [above of=ulsummer, shift={(0.35,0)}] {};
   \node [coordinate] (ull) [above of=ulsummer, shift={(-0.35,0)}] {};
   \node [coordinate] (right) [right of=ulm] {};

   \draw (summer.left in) .. controls +(120:0.2) and +(270:0.1) .. (ulsummer.io)
         (ulsummer.left in) .. controls +(120:0.35) and +(270:0.3) .. (ull)
         (ulsummer.right in) .. controls +(60:0.35) and +(270:0.3) .. (ulm)
         (summer.right in) .. controls +(60:0.75) and +(270:0.75) .. (right)
         (summer.io) -- +(270:0.5);
   \end{tikzpicture}
        \hspace{1.0cm}
   \begin{tikzpicture}[-, thick, node distance=0.7cm]
   \node [plus] (twadder) {};
   \node [coordinate] (twout) [below of=twadder] {};
   \node [coordinate] (twR) [above right of=twadder, shift={(-0.2,0)}] {};
   \node (cross) [above of=twadder] {};
   \node [coordinate] (twRIn) [above left of=cross, shift={(0,0.3)}] {};
   \node [coordinate] (twLIn) [above right of=cross, shift={(0,0.3)}] {};

   \draw (twadder.right in) .. controls +(60:0.35) and +(-45:0.25) .. (cross)
                            .. controls +(135:0.2) and +(270:0.4) .. (twRIn);
   \draw (twadder.left in) .. controls +(120:0.35) and +(-135:0.25) .. (cross.center)
                           .. controls +(45:0.2) and +(270:0.4) .. (twLIn);
   \draw (twout) -- (twadder);

   \node (eq) [right of=twR] {\(=\)};

   \node [coordinate] (L) [right of=eq] {};
   \node [plus] (adder) [below right of=L] {};
   \node [coordinate] (out) [below of=adder] {};
   \node [coordinate] (R) [above right of=adder] {};
   \node (cross) [above left of=R] {};
   \node [coordinate] (LIn) [above left of=cross] {};
   \node [coordinate] (RIn) [above right of=cross] {};

   \draw (adder.left in) .. controls +(120:0.7) and +(270:0.7) .. (LIn)
         (adder.right in) .. controls +(60:0.7) and +(270:0.7) .. (RIn)
         (out) -- (adder);
   \end{tikzpicture}
    }
\end{center}

We define a {\it monoid homomorphism} to be a function $f: A \rightarrow B$ between monoids that commutes with the multiplication and unit maps; that is, $\mu_B (f \otimes f) = f \mu_A$ and $f \eta_A = \eta_B$. This is represented pictorally as

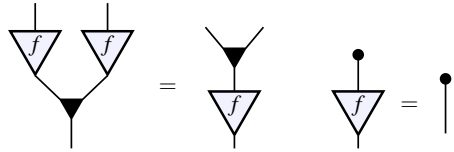
\begin{center}
    \scalebox{0.80}{
   \begin{tikzpicture}[-, thick, node distance=0.85cm]
   \node [plus] (adder) {};
   \node (out) [below of=adder] {};
   \node [multiply] (L) [above left of=adder, shift={(0,0.4)}] {\(f\)};
   \node [multiply] (R) [above right of=adder, shift={(0,0.4)}] {\(f\)};
   \node (RIn) [above of=R] {};
   \node (LIn) [above of=L] {};

   \draw (adder.right in) -- (R.io) (R) -- (RIn);
   \draw (adder.left in) -- (L.io) (L) -- (LIn);
   \draw (out) -- (adder);
   \end{tikzpicture}
        \hspace{0.1cm}
   \begin{tikzpicture}[node distance=1.15cm]
   \node (eq){\(=\)};
   \node [below of=eq] {};
   \end{tikzpicture}
   \begin{tikzpicture}[-, thick, node distance=0.85cm]
   \node (out) {};
   \node [multiply] (c) [above of=out] {\(f\)};
   \node [plus] (adder) [above of=c] {};
   \node (L) [above left of=adder] {};
   \node (R) [above right of=adder] {};

   \draw (R) -- (adder.right in) (adder) -- (c) -- (out);
   \draw (L) -- (adder.left in);
   \end{tikzpicture}
        \hspace{0.7cm}
   \begin{tikzpicture}[-, thick, node distance=0.85cm]
   \node [multiply] (prod) {\(f\)};
   \node (out0) [below of=prod] {};
   \node [zero] (ins0) [above of=prod] {};
   \node (eq) [right of=prod] {\(=\)};
   \node (out1) [below right of=eq] {};
   \node [zero] (ins1) [above of=out1, shift={(0,0.2)}] {};

   \draw (out0) -- (prod) -- (ins0);
   \draw (out1) -- (ins1);
   \end{tikzpicture}
       \hspace{0.7cm}
    }
\end{center}

Note that the composition of two monoid homomorphisms is a monoid homomorphism.

A {\it cocommutative comonoid} $A \in \Cc$ has the maps comultiplication $\Delta_A: A \rightarrow A \otimes A$ and the counit $\epsilon_A: A \rightarrow I$ represented as

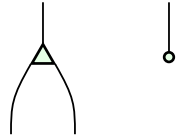
\begin{center}
\scalebox{0.85}{
   \begin{tikzpicture}[thick]
   \node[delta] (dupe){};
   \node (o1) at (-0.5,-1.35) {};
   \node (o2) at (0.5,-1.35) {};
   \node (in) [above of=dupe] {};

   \draw (o1) .. controls +(90:0.6) and +(-120:0.6) .. (dupe.left out);
   \draw (o2) .. controls +(90:0.6) and +(-60:0.6) .. (dupe.right out);
   \draw (in) -- (dupe);
   \end{tikzpicture}
}
\qquad
\scalebox{0.85}{
   \begin{tikzpicture}[thick]
   \node at (0,-2.35) {};
   \node (in1) {};
   \node [bang] (del1) at (0,-1) {};

   \draw (in1) -- (del1);
   \end{tikzpicture}
}
\quad
\end{center}

These maps satisfy inverted versions of the axioms for a commutative monoid:

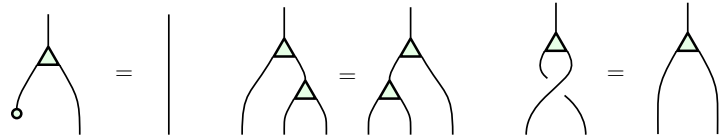
\begin{center}
    \scalebox{0.80}{
   \begin{tikzpicture}[-, thick, node distance=0.74cm]
   \node [delta] (dupe) {};
   \node [coordinate] (top) [above of=dupe] {};
   \node [coordinate] (Ldub) [below left of=dupe] {};
   \node [bang] (delete) [below of=Ldub, shift={(0,0.35)}] {};
   \node [coordinate] (Rdub) [below right of=dupe] {};
   \node [coordinate] (dubout) [below of=Rdub] {};
   \node (equal) [right of=Rdub, shift={(0,0.26)}] {\(=\)};
   \node [coordinate] (in) [right of=equal, shift={(0,1)}] {};
   \node [coordinate] (out) [right of=equal, shift={(0,-1)}] {};

   \draw (delete) .. controls +(90:0.3) and +(240:0.3) .. (dupe.left out)
         (dupe.right out) .. controls +(300:0.6) and +(90:0.6) .. (dubout)
         (dupe) -- (top)    (in) -- (out);
   \end{tikzpicture}
       \hspace{1.0cm}
   \begin{tikzpicture}[-, thick, node distance=0.7cm]
   \node [delta] (lrduper) {};
   \node [delta] (duper) [above of=lrduper, shift={(-0.35,0)}] {};
   \node [coordinate](lrm) [below of=lrduper, shift={(-0.35,0)}] {};
   \node [coordinate](lrr) [below of=lrduper, shift={(0.35,0)}] {};
   \node [coordinate](left) [left of=lrm] {};

   \draw (duper.right out) .. controls +(300:0.2) and +(90:0.1) .. (lrduper.io)
         (lrduper.right out) .. controls +(300:0.35) and +(90:0.3) .. (lrr)
         (lrduper.left out) .. controls +(240:0.35) and +(90:0.3) .. (lrm)
         (duper.left out) .. controls +(240:0.75) and +(90:0.75) .. (left)
         (duper.io) -- +(90:0.5);

   \node (eq) [right of=lrduper, shift={(0,0.25)}] {\(=\)};

   \node [delta] (lldubber) [right of=eq, shift={(0,-0.25)}] {};
   \node [delta] (dubber) [above of=lldubber, shift={(0.35,0)}] {};
   \node [coordinate] (llm) [below of=lldubber, shift={(0.35,0)}] {};
   \node [coordinate] (lll) [below of=lldubber, shift={(-0.35,0)}] {};
   \node [coordinate] (right) [right of=llm] {};

   \draw (dubber.left out) .. controls +(240:0.2) and +(90:0.1) .. (lldubber.io)
         (lldubber.left out) .. controls +(240:0.35) and +(90:0.3) .. (lll)
         (lldubber.right out) .. controls +(300:0.35) and +(90:0.3) .. (llm)
         (dubber.right out) .. controls +(300:0.75) and +(90:0.75) .. (right)
         (dubber.io) -- +(90:0.5);
   \end{tikzpicture}
       \hspace{1.0cm}
   \begin{tikzpicture}[-, thick, node distance=0.7cm]
   \node [coordinate] (twtop) {};
   \node [delta] (twdupe) [below of=twtop] {};
   \node [coordinate] (twR) [below right of=twdupe, shift={(-0.2,0)}] {};
   \node (cross) [below of=twdupe] {};
   \node [coordinate] (twROut) [below left of=cross, shift={(0,-0.3)}] {};
   \node [coordinate] (twLOut) [below right of=cross, shift={(0,-0.3)}] {};

   \draw (twdupe.left out) .. controls +(240:0.35) and +(135:0.25) .. (cross)
                           .. controls +(-45:0.2) and +(90:0.4) .. (twLOut)
         (twdupe.right out) .. controls +(300:0.35) and +(45:0.25) .. (cross.center)
                            .. controls +(-135:0.2) and +(90:0.4) .. (twROut)
         (twtop) -- (twdupe);

   \node (eq) [right of=twR] {\(=\)};

   \node [coordinate] (L) [right of=eq] {};
   \node [delta] (dupe) [above right of=L] {};
   \node [coordinate] (top) [above of=dupe] {};
   \node [coordinate] (R) [below right of=dupe] {};
   \node (uncross) [below left of=R] {};
   \node [coordinate] (LOut) [below left of=uncross] {};
   \node [coordinate] (ROut) [below right of=uncross] {};

   \draw (dupe.left out) .. controls +(240:0.7) and +(90:0.7) .. (LOut)
         (dupe.right out) .. controls +(300:0.7) and +(90:0.7) .. (ROut)
         (top) -- (dupe);
   \end{tikzpicture}
    }
\end{center}

We can also define a {\it comonoid homomorphism}:

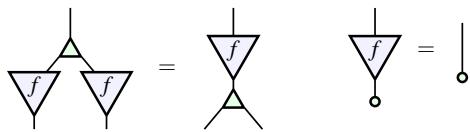
\begin{center}
    \scalebox{0.80}{
   \begin{tikzpicture}[-, thick, node distance=0.85cm]
   \node (top) {};
   \node [delta] (dupe) [below of=top] {};
   \node [multiply] (L) [below left of=dupe] {\(f\)};
   \node [multiply] (R) [below right of=dupe] {\(f\)};
   \node (ROut) [below of=R] {};
   \node (LOut) [below of=L] {};

   \draw[rounded corners] (dupe.left out) -- (L) -- (LOut);
   \draw[rounded corners] (dupe.right out) -- (R) -- (ROut);
   \draw (top) -- (dupe);
   \end{tikzpicture}
        \hspace{0.1cm}
   \begin{tikzpicture}[node distance=1.15cm]
   \node (eq){\(=\)};
   \node [below of=eq] {};
   \end{tikzpicture}
   \begin{tikzpicture}[-, thick, node distance=0.85cm]
   \node (top) {};
   \node [multiply] (c) [below of=top] {\(f\)};
   \node [delta] (dupe) [below of=c] {};
   \node (L) [below left of=dupe] {};
   \node (R) [below right of=dupe] {};

   \draw[rounded corners] (dupe.left out) -- (L);
   \draw[rounded corners] (dupe.right out) -- (R);
   \draw (top) -- (c) -- (dupe);
   \end{tikzpicture}
       \hspace{0.7cm}
   \begin{tikzpicture}[-, thick, node distance=0.85cm]
   \node [multiply] (prod) {\(f\)};
   \node (in0) [above of=prod] {};
   \node [bang] (del0) [below of=prod] {};
   \node (eq) [right of=prod] {\(=\)};
   \node (in1) [above right of=eq] {};
   \node [bang] (del1) [below of=in1, shift={(0,-0.2)}] {};
   \node [below left of=del0] {};

   \draw (in0) -- (prod) -- (del0);
   \draw (in1) -- (del1);
   \end{tikzpicture}
   } 
  \end{center}

Comonoid homomorphisms are again closed under composition.

A {\it bicommutative bimonoid} $A$ is a commutative monoid and a cocommutative comonoid such that certain distributive laws hold; namely, that maps defining the monoid structure commute with the maps defining the comonoid structure, so that $\mu_A \Delta_A = (\Delta_A \otimes \Delta_A) (1_A \otimes \tau \otimes 1_A) (\mu_A \otimes \mu_A)$, $\Delta_A \eta_A = \eta_A \otimes \eta_A$, $\epsilon_A \mu_A = \epsilon_A \otimes \epsilon_A$, and $\epsilon_A \mu_A = 1_I$. Pictorally, we have the following relations:
  \begin{center}
\scalebox{0.8}
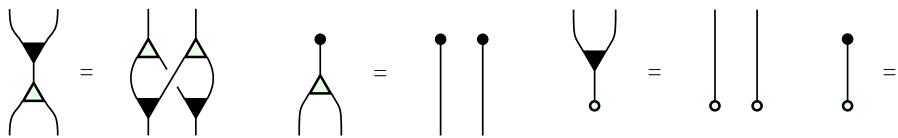
{
   \begin{tikzpicture}[thick]
   \node[plus] (adder) {};
   \node [coordinate] (f) [above of=adder, shift={(-0.4,-0.325)}] {};
   \node [coordinate] (g) [above of=adder, shift={(0.4,-0.325)}, label={}] {};
   \node[delta] (dupe) [below of=adder, shift={(0,0.25)}] {};
   \node [coordinate] (outL) [below of=dupe, shift={(-0.4,0.325)}, label={[shift={(-0.2,-0.6)}]}] {};
   \node [coordinate] (outR) [below of=dupe, shift={(0.4,0.325)}, label={[shift={(0.15,-0.6)}]}] {};

   \draw (adder.io) -- (dupe.io)
         (f) .. controls +(270:0.4) and +(120:0.25) .. (adder.left in)
         (adder.right in) .. controls +(60:0.25) and +(270:0.4) .. (g)
         (dupe.left out) .. controls +(240:0.25) and +(90:0.4) .. (outL)
         (dupe.right out) .. controls +(300:0.25) and +(90:0.4) .. (outR);
   \end{tikzpicture}
      \raisebox{4em}{=}
      \hspace{1em}
   \begin{tikzpicture}[-, thick, node distance=0.7cm]
   \node [plus] (addL) {};
   \node (cross) [above right of=addL, shift={(-0.1,-0.0435)}] {};
   \node [plus] (addR) [below right of=cross, shift={(-0.1,0.0435)}] {};
   \node [delta] (dupeL) [above left of=cross, shift={(0.1,-0.0435)}] {};
   \node [delta] (dupeR) [above right of=cross, shift={(-0.1,-0.0435)}] {};
   \node [coordinate] (f) [above of=dupeL] {};
   \node [coordinate] (g) [above of=dupeR] {};
   \node [coordinate] (sum1) [below of=addL, shift={(0,0.2)}, label={[shift={(-0.2,-0.6)}]}] {};
   \node [coordinate] (sum2) [below of=addR, shift={(0,0.2)}, label={[shift={(0.15,-0.6)}]}] {};

   \path
   (addL) edge (sum1) (addL.right in) edge (dupeR.left out) (addL.left in) edge [bend left=30] (dupeL.left out)
   (addR) edge (sum2) (addR.left in) edge (cross) (addR.right in) edge [bend right=30] (dupeR.right out)
   (dupeL) edge (f)
   (dupeL.right out) edge (cross)
   (dupeR) edge (g);
   \end{tikzpicture}
 \qquad\quad
   \begin{tikzpicture}[thick]
   \node [zero] (z) at (0,1) {};
   \node [delta] (dub) at (0,0.2) {};
   \node [coordinate] (oL) at (-0.35,-0.6) [label={[shift={(0,-0.6)}]}] {};
   \node [coordinate] (oR) at (0.35,-0.6) [label={[shift={(0,-0.6)}]}] {};

   \node (eq) at (1,0.42) {=};

   \node [zero] (zleft) at (2,1) {};
   \node [zero] (zright) at (2.7,1) {};
   \node [coordinate] (Lo) at (2,-0.6) [label={[shift={(0,-0.6)}]}] {};
   \node [coordinate] (Ro) at (2.7,-0.6) [label={[shift={(0,-0.6)}]}] {};

   \draw (z) -- (dub)
         (dub.left out) .. controls +(240:0.3) and +(90:0.5) .. (oL)
         (dub.right out) .. controls +(300:0.3) and +(90:0.5) .. (oR)
         (zleft) -- (Lo)
         (zright) -- (Ro)
   ;
   \end{tikzpicture}
 \qquad\quad
   \begin{tikzpicture}[thick]
   \node [bang] (b) at (0,-1) {};
   \node [plus] (sum) at (0,-0.2) {};
   \node [coordinate] (oL) at (-0.35,0.6) [label={}] {};
   \node [coordinate] (oR) at (0.35,0.6) [label={}] {};
   \node (spacemaker) at (0,-1.95) {};

   \node (eq) at (1,-0.47) {=};

   \node [bang] (bleft) at (2,-1) {};
   \node [bang] (bright) at (2.7,-1) {};
   \node [coordinate] (Lo) at (2,0.6) [label={}] {};
   \node [coordinate] (Ro) at (2.7,0.6) [label={}] {};

   \draw (b) -- (sum)
         (sum.left in) .. controls +(120:0.3) and +(270:0.5) .. (oL)
         (sum.right in) .. controls +(60:0.3) and +(270:0.5) .. (oR)
         (bleft) -- (Lo)
         (bright) -- (Ro)
   ;
   \end{tikzpicture}
 \qquad\quad
   \begin{tikzpicture}[thick]
   \node [zero] (z) at (0,0.11) {};
   \node [bang] (b) at (0,-1) {};
   \node (spacemaker) at (0,-1.95) {};

   \node (eq) at (0.7,-0.47) {=};

   \draw (z) -- (b)
   ;
   \end{tikzpicture}
}
  \end{center} 

Note that the above axioms are equivalent to saying that $\Delta_A$ and $\epsilon_A$ are monoid homomorphisms, or that $\mu_A$ and $\eta_A$ are comonoid homomorphisms.
Bimonoid homomorphisms $f: A \rightarrow B$ are morphisms that are both monoid and comonoid homomorphisms; they, too, are closed under composition.
If $\Cc$ consists of $k$-vector spaces and $k$-linear maps for some field $k$, then we call the above constructions {\it commutative algebras}, {\it cocommutative coalgebras}, and {\it bicommutative bialgebras} respectively.

Given two bicommutative bimonoids $A, B \in \Cc$, we can give $A \otimes B$ the structure of a bicommutative bimonoid by setting $\mu_{A \otimes B} = (\mu_A \otimes \mu_B) (1_A \otimes \tau_{BA} \otimes 1_B)$, $\eta_{A \otimes B} = \eta_A \otimes \eta_B$, $\Delta_{A \otimes B} = (1_A \otimes \tau{BA} \otimes 1_B)(\Delta_A \otimes \Delta_B)$, and $\epsilon_{A \otimes B} = \epsilon_A \otimes \epsilon_B$. In this way we may make $A^{\otimes m}$ into a bicommutative bimonoid for each $m \geq 2$. Since the symmetry $\tau$ is natural, the bimonoid homomorphisms of $A \otimes B \rightarrow A' \otimes B'$ are exactly $f \otimes g$ where $f: A \rightarrow A'$ and $g: B \rightarrow B'$ are bimonoid homomorphisms.

We will write $\mu_A^n$ for the functor $A^{\otimes n} \rightarrow A$ given inductively by $\mu_A^0 = \eta_A$ and $\mu_A^{n+1} = \mu_A (1 \otimes \mu_A^n)$ and $\Delta_A^n$ for the functor $A \rightarrow A^{\otimes n}$ given inductively by $\Delta_A^0 = \epsilon_A$ and $\Delta_A^{n+1} = (1_A \otimes \Delta_A^n)\Delta_A$. By associativity, $\mu_A^n$ can be depicted as any tree with $n$ leaves; coassociativity means a similar statement holds for $\Delta_A^n$.

We define $\Bimon$ to be the subcategory of $\Cc$ whose objects are bicommutative bimonoids and whose morphisms are the bimonoid homomorphisms between them.

\begin{lemma} \label{enriched}
The category $\Bimon$ may be enriched over commutative monoids.
\end{lemma}
\begin{proof}
If $f, g \in \Hom_{\Bimon} (A,B)$, define $f + g$ to be $\mu_B(f \otimes g)\Delta_A$; note that this is a composition of bimonoid homomorphisms and is hence a bimonoid homomorphism. We need to show that addition defines a commutitive monoid structure on $\Hom_{\Bimon} (A,B)$ and that the composition map $\Hom_{\Bimon} (B,C) \otimes \Hom_{\Bimon} (A,B) \rightarrow \Hom_{\Bimon} (A,C)$ is a monoid homomorphism. For the former, associativity occurs since 
\begin{align*}
  (f+g)+h &= \mu_B((\mu_B(f \otimes g) \Delta_A) \otimes h ) \Delta_A \\
  &= \mu_B (\mu_B \otimes 1)(f \otimes g \otimes h) (\Delta_A \otimes 1) \Delta_A \\
  &= \mu_B (1 \otimes \mu_B)(f \otimes g \otimes h)(1 \otimes \Delta_A)\Delta_A \\
  &= \mu_B (f \otimes (\mu_B (g \otimes h) \Delta_A) \Delta_A = f+(g+h),
\end{align*}
where the third equality follows from the associativity of $B$ and the coassociativity of $A$. The zero element of $\Hom_{\Bimon} (A,B)$ is $0 = \epsilon_B \eta_A$; the unit laws of $\Hom_{\Bimon} (A,B)$ follow from the unit laws of $B$ and the counit laws of $A$, for instance
\begin{align*}
  0+f &= \mu_B ((\epsilon_B \eta_A) \otimes f) \Delta_A \\
  &= \mu_B (\epsilon_B \otimes 1) f (\eta_A \otimes 1) \Delta_A = f.
\end{align*}
To check commutativity, let $\tau_A: A \otimes A \rightarrow A \otimes A$ be the symmetry map. Then, by the commutativity of $B$ and the cocommutativity of $A$,
\begin{align*}
  f+g &= \mu_B (f \otimes g) \Delta_A = \mu_B \tau_B (f \otimes g) \Delta_A \\
  &= \mu_B (g \otimes f) \tau_A \Delta_A = \mu_B (g \otimes f) \Delta_A = g+f.
\end{align*}
We now show that the composition map $\Hom_{\Bimon} (B,C) \times \Hom_{\Bimon} (A,B) \rightarrow \Hom_{\Bimon} (A,C)$ is bilinear, so it induces a monoid map $\Hom_{\Bimon} (B,C) \otimes \Hom_{\Bimon} (A,B) \rightarrow \Hom_{\Bimon} (A,C)$. If $f_1, f_2 \in \Hom_{\Bimon} (B,C)$ and $g_1,g_2 \in \Hom_{\Bimon} (A,B)$, then
\begin{align*}
  f_1(g_1+g_2) &= f_1 \mu_B (g_1 \otimes g_2) \Delta_A \\
  &= \mu_C (f_1 \otimes f_1)(g_1 \otimes g_2) \Delta_A \\
  &= \mu_C (f_1 g_1 \otimes f_1 g_2) \Delta_A = f_1 g_1 + f_1 g_2,
\end{align*}
and similarly $(f_1 + f_2)g_1 = f_1 g_1 + f_2 g_1$.
\end{proof}

Using this enriched structure as addition and composition as multiplication, we can make $\Hom_{\Bimon}(A,A)$ into a rig for all $A \in \Bimon$. We will call this rig $\End(A)$.

For an arbitrary commutative rig $R$, we define a category $\Bimon^R$ whose objects are pairs $(A,\phi_A)$ where $A$ is a bimonoid in $\Cc$ and where $\phi_A$ is a rig homomorphism $R \rightarrow \End(A)$. The morphisms $f: (A, \phi_A) \rightarrow (B, \phi_B)$ are bimonoid maps $f: A \rightarrow B$ for which the diagram
\begin{center} \begin{tikzpicture}
  \matrix (m) [matrix of math nodes,row sep=3em,column sep=4em,minimum width=2em]
  {
     A & B \\
     A & B \\};
  \path[-stealth]
    (m-1-1) edge node [left] {$\phi_A(r)$} (m-2-1)
            edge node [above] {$f$} (m-1-2)
    (m-2-1) edge node [below] {$f$} (m-2-2)
    (m-1-2) edge node [right] {$\phi_B(r)$} (m-2-2);
\end{tikzpicture} \end{center}
commutes for all $r \in R$.

\section{The PROP $\Mat(R)$}

For a given commutative rig $R$, we define $\Mat(R)$ to be the PROP where $\Hom_{\Mat(R)} (m,n)$ consists of $n \times m$ R-valued matrices and where composition is defined by matrix multiplication. The tensor product of two matrices $f \otimes g$ is given by the block diagonal matrix
\[
  \begin{pmatrix} f & 0 \\ 0 & g \end{pmatrix}.
\]
Note that, since there is exactly one matrix with $n$ rows and no columns and that its transpose is the unique matrix with no rows and $n$ columns, $0$ is both the initial and terminal object in $\Mat(R)$. Also note that the symmetry is given by
\[
  \begin{pmatrix} 0 & 1 \\ 1 & 0 \end{pmatrix}.
\]
If $f$ and $g$ are $m \times n$ matrices, we can define an enriched structure on $\Mat(R)$ (see Lemma \ref{enriched} below) by letting $f+g$ be the usual matrix addition. 

We can also equip the object $1 \in \Mat(R)$ with a bicommutative bimonoid structure where (abusing notation slightly) $\mu_R = \begin{pmatrix} 1 & 1 \end{pmatrix}$ and $\eta_R = 0: 0 \rightarrow 1$ are the multiplication and unit maps. The comultiplication $\Delta_R$ and counit $\epsilon_R$ are defined by the transposes of these matrices. Checking the axioms now amounts to multiplying matrices; for instance, associativity holds since
\[
  \mu_R (1 \otimes \mu_R) = \begin{pmatrix} 1 & 1 \end{pmatrix}
  \begin{pmatrix} 1 & 0 & 0 \\ 0 & 1 & 1 \end{pmatrix} 
  = \begin{pmatrix} 1 & 1 \end{pmatrix}
  \begin{pmatrix} 1 & 1 & 0 \\ 0 & 0 & 1 \end{pmatrix}
  = \mu_R (\mu_R \otimes 1).
\]
Transposing matrices gives a self-inverse contravariant endofunctor on $\MR$; this functor allows us to extend results shown for the monoid $(\MR,\mu_R,\eta_R)$ to the comonoid $(\MR,\Delta_R,\epsilon_R)$.

\begin{lemma} \label{gen}
For any rig $R$, the object $1$ together with the morphisms $\mu_R$, $\eta_R$, $\Delta_R$, and $\epsilon_R$ defined above as well as the map $r: 1 \rightarrow 1$ for every $r \in R$ generate $\MR$.
\end{lemma}
\begin{proof}
Any object $n \in \MR$ is equal to $1^{\otimes n}$.

If $(r_i)$ is a $1 \times n$ matrix, then $(r_i) = \mu_R^{n}((r_1) \otimes \cdots \otimes (r_n))$, since
\[
  \begin{pmatrix} 1 & \cdots & 1 \end{pmatrix} 
  \begin{pmatrix} r_1 & & 0 \\ & \ddots & \\ 0 & & r_n \end{pmatrix}
  = \begin{pmatrix} r_1 & \cdots & r_n \end{pmatrix}.
\]
If $(r_{ij})$ is an arbitrary $m \times n$ matrix, we have that $(r_{ij}) = ((r_{i1})_i \otimes \cdots \otimes (r_{in})_i)\Delta_R^m$, where $(r_{ij})_i$ is the $j$th column of the matrix $(r_{ij})$, since
\[
  \begin{pmatrix} (r_{i1})_i & & 0 \\ & \ddots & \\ & & (r_{in})_i \end{pmatrix}
  \begin{pmatrix} 1 \\ \vdots \\ 1 \end{pmatrix}
  = \begin{pmatrix} (r_{i1})_i \\ \vdots \\ (r_{in})_i \end{pmatrix}
\]
\end{proof}

\section{Algebras over $\Mat(R)$}

\begin{lemma} \label{fwd}
If $A$ is a bicommutative bimonoid, $R$ is a commutative rig, and $\phi: R \rightarrow \End(A)$ is a map of rigs, then there is a unique strict monoidal functor $F_A$ enriched over commutative monoids from $\Mat(R)$ to $\Bimon$ such that
\begin{enumerate}[label=(\arabic*)]
  \item $A = F_A (1)$
  \item $\mu_A = F_A(\mu_R)$
  \item $\eta_A = F_A(\eta_R)$
  \item $\Delta_A = F_A(\Delta_R)$
  \item $\epsilon_A = F_A(\epsilon_R)$
  \item $\phi(r) = F_A(r: 1 \rightarrow 1)$ for all $r \in R$.
\end{enumerate}
This functor is necessarily symmetric.
\end{lemma}
\begin{proof}
By Lemma \ref{gen} and by the fact that we wish $F_A$ to be a strict monoidal functor, the given information shows how $F_A$ behaves on generators and is therefore enough to define $F_A$ on all of $\MR$. We are required to set $F_A(1) = A$; on all other objects, the monoidal structre of $F_A$ gives us $F_A(n) = A^{\otimes n}$. Note that $F_A$ is also determined on morphisms; in particular, if $(r_i)$ is a $1 \times n$ matrix, define $F_A(r_i) = \mu_A(\phi(r_1) \otimes \cdots \otimes \phi(r_n))$, and if $(r_{ij})$ is an arbitrary $m \times n$ matrix, define $F_A(r_{ij}) = (F_A((r_{i1})_i) \otimes \cdots \otimes F_A((r_{im})_i))\Delta_{A^{\otimes n}}^m$. Note that $F_A(r_{ij})$ is a composite of bimonoid maps, and is thus itself a bimonoid map. Note also that, since we have stipulated its values on the generators of $\MR$, $F_A$ is necessarily unique.

To see that $F_A$ preserves the enriched structure l, first let $(r_i)$ and $(s_i)$ be $1 \times n$ matrices. Then
\begin{align*}
  F_A(r_i + s_i) &= \mu_A^n(\phi(r_1+s_1) \otimes \cdots \otimes \phi(r_n+s_n)) \\
  &= \mu_A^n((\phi(r_1)+\phi(s_1)) \otimes \cdots \otimes (\phi(r_n)+\phi(s_n))) \\
  &= \mu_A^n(\mu_A(\phi(r_1) \otimes \phi(s_1))\Delta_A \otimes \cdots \otimes
    \mu_A(\phi(r_n) \otimes \phi(s_n))\Delta_A) \\
  &= \mu_A^n\mu_{A^{\otimes n}}(\phi(r_1) \otimes \phi(s_1) \otimes \cdots \otimes
    \phi(r_n) \otimes \phi(s_n)) \Delta_A \\
  &= \mu_A^n (\mu_A^n(\phi(r_1) \otimes \cdots \otimes \phi(r_n)) \otimes
    \mu_A^n(\phi(s_1) \otimes \cdots \otimes \phi(s_n)))\Delta_{A^{\otimes n}} \\
  &= F_A(r_i) + F_A(s_i).
\end{align*}
For the general case, assume $(r_{ij})$ and $(s_{ij})$ are $m \times n$ matrices. Then
\begin{align*}
  F_A(r_{ij}+s_{ij}) &= (F_A(r_{i1}+s_{i1}) \otimes \cdots \otimes F_A(r_{im}+s_{im}))
    \Delta_{A^{\otimes n}}^m \\
  &= ((F_A(r_{i1}) + F_A(s_{i1})) \otimes \cdots \otimes (F_A(r_{im}+s_{im}))) 
    \Delta_{A^{\otimes n}}^m \\
  &= (\mu_{A^{\otimes n}}(F_A(r_{i1}) \otimes F_A(s_{i1}))\Delta_{A^{\otimes n}} \otimes
    \cdots \otimes \mu_{A^{\otimes n}}(F_A(r_{i1}) \otimes F_A(s_{i1}))
    \Delta_{A^{\otimes n}})\Delta_{A^{\otimes n}}^m \\
  &= \mu_{A^{\otimes m}}((F_A(r_{i1}) \otimes F_A(r_{im}))\Delta_{A^{\otimes n}}^m
    \otimes (F_A(s_{i1}) \otimes F_A(s_{im}))\Delta_{A^{\otimes n}}^m)
    \Delta_{A^{\otimes n}} \\
  &= \mu_{A^{\otimes m}} (F_A(r_{ij}) \otimes F_A(s_{ij}))\Delta_{A^{\otimes n}} \\
  &= F_A(r_{ij}) + F_A(s_{ij}).
\end{align*}

We now show $F_A$ is a functor. Because we already checked that $F_A$ preserves the enriched structure, it suffices to show that $F_A$ preserves composition for matrices of the form
\[
  E_{ij}^{mn}(r) = 0^{i-1 \times j-1} \otimes (r) \otimes 0^{m-i \times n-j},
\]
since every matrix is the sum of matrices which are zero everywhere except at one place. Note that $F_A(E_{1j}^{1n}(r)) = \epsilon_{A^{\otimes j-1}} \otimes \phi(r) \otimes \epsilon_{A^{\otimes n-j}}$, so
\[
  F_A(E_{ij}^{mn}(r)) = \eta_{A^{\otimes j-1}} \otimes \epsilon_{A^{\otimes i-1}} 
  \otimes \phi(r) \otimes \epsilon_{A^{\otimes n-i}} \otimes \eta_{A^{\otimes m-j}}.
\]
To see that $F_A$ is symmetric, we have, by the naturality of $\tau_{AA}$,
\begin{align*}
  F_A(E^{22}_{12}(1) + E^{22}_{21}(1)) 
    &= (\eta_A \otimes 1_A \otimes \epsilon_A) + (\epsilon_A \otimes 1_A \otimes \eta_A)\\
  &= \mu_{A \otimes A}
    (\eta_A \otimes 1_A \otimes \epsilon_A \otimes \epsilon_A 
    \otimes 1_A \otimes \eta_A) \Delta_{A \otimes A} \\
  &= (\mu_A \otimes \mu_A)(1_A \otimes \tau_{AA} \otimes 1_A)
    (\eta_A \otimes 1_A \otimes \epsilon_A \otimes \epsilon_A 
    \otimes 1_A \otimes \eta_A) 
    (1_A \otimes \tau_{AA} \otimes 1_A)(\Delta_A \otimes \Delta_A) \\
  &= (\mu_A \otimes \mu_A)(\eta_A \otimes 1_A \otimes 1_A \otimes \eta_A) \tau_{AA}
    (1_A \otimes \eta_A \otimes \eta_A \otimes 1_A)(\Delta_A \otimes \Delta_A) \\
  &= (\mu_A(\eta_A \otimes 1_A) \otimes \mu_A(1_A \otimes \eta_A)) \tau_{AA}
    ((1_A \otimes \epsilon_A)\Delta_A \otimes (\epsilon_A \otimes 1_A)\Delta_A) \\
  &= \tau_A.
\end{align*}
\end{proof}

\begin{lemma} \label{rev}
If $R$ is a commutative rig and $F: \Mat(R) \rightarrow \Cc$ is an algebra over the PROP $\Mat(R)$, then there exists some $(A, \phi_A) \in \Bimon^R$ such that $A = F(1)$, $\mu_A = F(\mu_R)$, $\eta_A = F(\eta_R)$, $\Delta_A = F(\Delta_R)$, $\epsilon_A = F(\epsilon_R)$, and $\phi_A(r) = F(r: 1 \rightarrow 1)$. Furthermore, $F$ is an enriched functor.
\end{lemma}
\begin{proof}
The axioms for $A$ being a bicommutative bimonoid follow from the respective axioms of $\MR$ and the fact that $F$ is a monoidal functor.

We still need to show that $F$ is an enriched functor. Let $X$ and $Y$ be $m \times n$ matrices; then
\begin{align*}
  \mu_{A^{\otimes m}}(F(X) \otimes F(Y))\Delta_{A^{\otimes n}} 
    &= F \begin{pmatrix} I_m&I_m \end{pmatrix}
    F \begin{pmatrix} X&0 \\ 0&Y \end{pmatrix}
    F \begin{pmatrix} I_n \\ I_n \end{pmatrix} \\
  &= F(X+Y).
\end{align*}
\end{proof}

We now prove our main theorem:
\begin{thm} \label{main}
If $R$ is a commutative rig, then $\Mat(R)$ is the PROP for bicommutative bimonoid $A$ equipped with a map of rigs $\phi: R \rightarrow \End(A)$.
\end{thm}

\begin{proof}
We begin by giving a functor $\Fb$ from $\Bimon^R$ to the category of algebras over $\Mat(R)$. Given an element $(A, \phi_A) \in \Bimon^R$, let $\Fb(A,\phi_A)$ be the functor indicated by Lemma~\ref{fwd}; note in particular that $A = \Fb(A,\phi(A))(1)$ and $\phi(r) = \Fb(A,\phi_A)(r: 1 \rightarrow 1)$ for all $r \in R$. Suppose $f:(A,\phi_A) \rightarrow (B,\phi_B)$ is a morphism in $\Bimon^R$. We must define a natural transformation $\Fb(f)$ from $\Fb(A,\phi_A)$ to $\Fb(B,\phi_B)$; let $\Fb (f)(n) = f^{\otimes n}$. We need to show that for every $n \times m$ matrix $X$ with coefficients in $R$, the diagram
\begin{center} \begin{tikzpicture}
  \matrix (m) [matrix of math nodes,row sep=3em,column sep=6em,minimum width=2em]
  {
     A^{\otimes m} & A^{\otimes n} \\
     B^{\otimes m} & B^{\otimes n} \\};
  \path[-stealth]
    (m-1-1) edge node [left] {$f^{\otimes m}$} (m-2-1)
            edge node [above] {$F(A)(X)$} (m-1-2)
    (m-2-1) edge node [below] {$F(B)(X)$} (m-2-2)
    (m-1-2) edge node [right] {$f^{\otimes n}$} (m-2-2);
\end{tikzpicture} \end{center}
commutes.
Because of the enriched structure on the functors $\Fb (A, \phi_A)$ and $\Fb (B, \phi_B)$, it suffices to check this condition for matrices with precisely one nonzero entry. For $1 \times 1$ matrices $(r)$, this becomes $f\phi_A(r) = \phi_B(r)(f)$, which is true since $f$ is a morphism in $\Bimon^R$. It is also true when $m$ or $n$ is $0$ since $(\Fb(A, \phi_A))(0)$ is both an initial and terminal object. Since each matrix with one nonzero entry is a tensor product of matrices of these types, we know that $\Fb$ is a functor.

Now we want a functor $\Gb$ from the category of algebras over $\Mat(R)$ to $\Bimon^R$; letting $\Gb$ be the functor indicated by Lemma~\ref{rev}, we have $\Gb(F: \Mat(R) \rightarrow \Cc) = (F(1),F(\Hom_{\Mat(R)}(1,1)))$. Our functor sends the natural transformation $\alpha: F \rightarrow G$ to the component $\alpha_1$. This is well-defined since
\begin{center} \begin{tikzpicture}
  \matrix (m) [matrix of math nodes,row sep=3em,column sep=4em,minimum width=2em]
  {
     F(1) & F(1) \\
     G(1) & G(1) \\};
  \path[-stealth]
    (m-1-1) edge node [left] {$\alpha_1$} (m-2-1)
            edge node [above] {$F(r)$} (m-1-2)
    (m-1-2) edge node [right] {$\alpha_1$} (m-2-2)
    (m-2-1) edge node [below] {$G(r)$} (m-2-2);
\end{tikzpicture} \end{center}
commutes. We know $\Gb$ is a functor because composition of natural transformations is defined componentwise.

By the way we have defined $\Fb$ and $\Gb$, we can see that $\Gb \Fb = \id_{\Bimon^R}$ and $\Fb \Gb$ is the identity functor on the category of algebras of $\Mat(R)$, so we have an isomorphism of categories.
\end{proof}

The fact that $\phi: R \rightarrow \End(A)$ is a rig homomorphism has the following diagrammatic representation, letting a triangle containing the variable $r \in R$ represent $\phi(r)$:
\begin{center}
    \scalebox{0.80}{
   \begin{tikzpicture}[-, thick, node distance=0.85cm]
   \node (bctop) {};
   \node [multiply] (bc) [below of=bctop, shift={(0,-0.59)}] {\( \scriptstyle{r+s} \)};
   \node (bcbottom) [below of=bc, shift={(0,-0.59)}] {};

   \draw (bctop) -- (bc) -- (bcbottom);

   \node (eq) [right of=bc, shift={(0.15,0)}] {\(=\)};

   \node [multiply] (b) [right of=eq, shift={(0,0.1)}] {\(\scriptstyle{r}\)};
   \node [delta] (dupe) [above right of=b, shift={(-0.2,0)}] {};
   \node (top) [above of=dupe, shift={(0,-0.1)}] {};
   \node [multiply] (c) [below right of=dupe, shift={(-0.2,0)}] {\(\scriptstyle{s}\)};
   \node [plus] (adder) [below right of=b, shift={(-0.2,-0.2)}] {};
   \node (out) [below of=adder, shift={(0,0.1)}] {};

   \draw
   (dupe.left out) .. controls +(240:0.15) and +(90:0.15) .. (b.90)
   (dupe.right out) .. controls +(300:0.15) and +(90:0.15) .. (c.90)
   (top) -- (dupe.io)
   (adder.io) -- (out)
   (adder.left in) .. controls +(120:0.15) and +(270:0.15) .. (b.io)
   (adder.right in) .. controls +(60:0.15) and +(270:0.15) .. (c.io);
   \end{tikzpicture}
        \hspace{0.8cm}
   \begin{tikzpicture}[-, thick]
   \node (top) {};
   \node [multiply] (c) [below of=top] {\(s\)};
   \node [multiply] (b) [below of=c] {\(r\)};
   \node (bottom) [below of=b] {};

   \draw (top) -- (c) -- (b) -- (bottom);

   \node (eq) [left of=b, shift={(0.2,0.5)}] {\(=\)};

   \node (bctop) [left of=top, shift={(-0.6,0)}] {};
   \node [multiply] (bc) [left of=eq, shift={(0.2,0)}] {\(rs\)};
   \node (bcbottom) [left of=bottom, shift={(-0.6,0)}] {};

   \draw (bctop) -- (bc) -- (bcbottom);
   \end{tikzpicture}
        \hspace{0.8cm}
\raisebox{2em}{
   \begin{tikzpicture}[-, thick, node distance=0.85cm]
   \node (top) {};
   \node [multiply] (one) [below of=top] {1};
   \node (bottom) [below of=one] {};

   \draw (top) -- (one) -- (bottom);

   \node (eq) [right of=one] {\(=\)};
   \node (topid) [right of=top, shift={(0.6,0)}] {};
   \node (botid) [right of=bottom, shift={(0.6,0)}] {};

   \draw (topid) -- (botid);
   \end{tikzpicture}
}
        \hspace{0.6cm}
\raisebox{2em}{
   \begin{tikzpicture}[-, thick, node distance=0.85cm]
   \node [multiply] (prod) {\(0\)};
   \node (in0) [above of=prod] {};
   \node (out0) [below of=prod] {};
   \node (eq) [right of=prod] {\(=\)};
   \node [bang] (del) [right of=eq, shift={(-0.2,0.2)}] {};
   \node [zero] (ins) [right of=eq, shift={(-0.2,-0.2)}] {};
   \node (in1) [above of=del, shift={(0,-0.2)}] {};
   \node (out1) [below of=ins, shift={(0,0.2)}] {};

   \draw (in0) -- (prod) -- (out0);
   \draw (in1) -- (del);
   \draw (ins) -- (out1);
   \end{tikzpicture}
    }
}
\end{center}
These diagrams, together with those in Section~\ref{bimon}, constitute the entire set of relations given by Baez and Erbele for $\Mat(k)$, where $k$ is a field \cite{Erb}.

\section{Examples}

As noted earlier, $\Mat(R)$ is equivalent to some famous PROPs for the right choices of $R$. The only possible two-element rigs are the Boolean rig $\B$ where $1 + 1 = 1$ and $\mathbb F_2$ where $1 + 1 = 0$. We define $\FinRel$ to be the category whose objects are finite sets and whose morphisms are relations between them. Because all sets of the same cardinality are isomorphic in this category, we can make $\FinRel$ into a PROP by choosing one representative of each cardinality and defining the tensor product on objects to be the disjoint union of sets. A relation from $m$ to $n$ can be defined as a Boolean-valued matrix by letting the $i$th row and $j$th column be 1 if and only if $i$ relates to $j$. Relations can also be visualized as string diagrams, where $i$ relates to $j$ if and only if a path connects the two; an example is given in Figure~\ref{rlnstring}.

\begin{figure}[h]
\begin{center}
\scalebox{0.8}{
\begin{tikzpicture}[thick]


\node (xa) at (0,0) {$a$};
\path (xa) ++(0:1.7) node (xb) {$b$};
\path (xb) ++(0:1) node (xc) {$c$};
\path (xc) ++(0:1) node (xd) {$d$};
\path (xd) ++(0:0.7) node (xe) {$e$};
\path (xe) ++(0:0.7) node (xf) {$f$};

\path (xe) ++(-90:0.8) node [bang] (ee) {};

\path (xa) ++(-90:1.5) node [delta] (da1) {};
\path (xb) ++(-90:1.5) node [delta] (db) {};
\path (xd) ++(-90:1.5) node [delta] (dd1) {};

\path (da1.left out) ++(-100:0.5) node [delta] (da2) {};
\path (dd1.left out) ++(-100:0.5) node [delta] (dd2) {};

\path (xc) ++(-90:3) node [coordinate] (cc) {};
\path (xf) ++(-90:3) node [coordinate] (cf) {};

\path (da1) ++(-0.65,-1.5) coordinate (ca1);
\path (da1) ++(0.125,-1.5) coordinate (ca2);
\path (da1) ++(0.65,-1.5) coordinate (ca3);

\path (db) ++(-0.65,-1.5) coordinate (cb1);
\path (db) ++(0.65,-1.5) coordinate (cb2);

\path (dd1) ++(-0.65,-1.5) coordinate (cd1);
\path (dd1) ++(0.125,-1.5) coordinate (cd2);
\path (dd1) ++(0.65,-1.5) coordinate (cd3);

\draw 
  (xa) -- (da1.io)
  (xb) -- (db.io)
  (xc) -- (cc)
  (xd) -- (dd1.io)
  (xe) -- (ee)
  (xf) -- (cf)
  (da1.left out) .. controls +(-120:0.1) and +(90:0.1) .. (da2.io)
  (dd1.left out) .. controls +(-120:0.1) and +(90:0.1) .. (dd2.io)
  (da2.left out) .. controls +(-120:0.1) and +(90:0.3) .. (ca1)
  (da2.right out) .. controls +(-60:0.1) and +(90:0.3) .. (ca2)
  (da1.right out) .. controls +(-60:0.4) and +(90:0.5) .. (ca3)
  (db.left out) .. controls +(-120:0.4) and +(90:0.5) .. (cb1)
  (db.right out) .. controls +(-60:0.4) and +(90:0.5) .. (cb2)
  (dd2.left out) .. controls +(-120:0.1) and +(90:0.3) .. (cd1)
  (dd2.right out) .. controls +(-60:0.1) and +(90:0.3) .. (cd2)
  (dd1.right out) .. controls +(-60:0.4) and +(90:0.5) .. (cd3)
  ;


\node (yu) at (-0.25,-11) {$u$};
\path (yu) ++(0:1.7) node (yv) {$v$};
\path (yv) ++(0:2.25) node (yw) {$w$};
\path (yw) ++(0:0.7) node (yx) {$x$};
\path (yx) ++(0:0.7) node (yy) {$y$};
\path (yy) ++(0:0.7) node (yz) {$z$};

\path (yx) ++(90:0.8) node [zero] (ix) {};
\path (yy) ++(90:0.8) node [zero] (iy) {};

\path (yu) ++(90:1.5) node [plus] (mu) {};
\path (yv) ++(90:1.5) node [plus] (mv) {};
\path (yw) ++(90:1.5) node [plus] (mw1) {};

\path (mw1.left in) ++(100:0.5) node [plus] (mw2) {};
\path (mw2.left in) ++(100:0.5) node [plus] (mw3) {};
\path (mw3.left in) ++(100:0.5) node [plus] (mw4) {};

\path (mu) ++(-0.65,2.5) coordinate (fu1);
\path (mu) ++(0.65,2.5) coordinate (fu2);

\path (mv) ++(-0.65,2.5) coordinate (fv1);
\path (mv) ++(0.65,2.5) coordinate (fv2);

\path (mw1) ++(-1.2,2.5) coordinate (fw1);
\path (mw1) ++(-0.5,2.5) coordinate (fw2);
\path (mw1) ++(-0.1,2.5) coordinate (fw3);
\path (mw1) ++(0.3,2.5) coordinate (fw4);
\path (mw1) ++(0.7,2.5) coordinate (fw5);

\path (yz) ++(0,4) coordinate (fz);

\draw 
  (yx) -- (ix)
  (yy) -- (iy)
  (yu) -- (mu.io)
  (yv) -- (mv.io)
  (yw) -- (mw1.io)
  (mu.left in) .. controls +(120:0.4) and +(-90:0.5) .. (fu1)
  (mu.right in) .. controls +(60:0.4) and +(-90:0.5) .. (fu2)
  (mv.left in) .. controls +(120:0.4) and +(-90:0.5) .. (fv1)
  (mv.right in) .. controls +(60:0.4) and +(-90:0.5) .. (fv2)
  (mw1.left in) .. controls +(120:0.1) and +(-90:0.1) .. (mw2.io)
  (mw2.left in) .. controls +(120:0.1) and +(-90:0.1) .. (mw3.io)
  (mw3.left in) .. controls +(120:0.1) and +(-90:0.1) .. (mw4.io)
  (mw4.left in) .. controls +(120:0.1) and +(-90:0.3) .. (fw1)
  (mw4.right in) .. controls +(60:0.1) and +(-90:0.3) .. (fw2)
  (mw3.right in) .. controls +(60:0.2) and +(-90:0.4) .. (fw3)
  (mw2.right in) .. controls +(60:0.35) and +(-90:0.6) .. (fw4)
  (mw1.right in) .. controls +(60:0.5) and +(-90:0.8) .. (fw5)
  (yz) -- (fz)
  ;


\draw 
  (cf) .. controls +(-90:1.5) and +(90:1.5) .. (fw5)
  (cd3) .. controls +(-90:1.5) and +(90:1.5) .. (fw4)
  (cd2) .. controls +(-90:1.5) and +(90:1.5) .. (fv2)
  (cd1) .. controls +(-90:1.5) and +(90:1.5) .. (fu2)
  ;

\draw[overdraw] (cc) .. controls +(-90:1.5) and +(90:1.5) .. (fw3);
\draw[overdraw] (cb2) .. controls +(-90:1.5) and +(90:1.5) .. (fw2);
\draw[overdraw] (cb1) .. controls +(-90:1.5) and +(90:1.5) .. (fv1);
\draw[overdraw] (ca3) .. controls +(-90:2) and +(90:2.5) .. (fz);
\draw[overdraw] (ca2) .. controls +(-90:1.5) and +(90:1.5) .. (fw1);
\draw[overdraw] (ca1) .. controls +(-90:1.5) and +(90:1.5) .. (fu1);

\path (yz) ++(1.5,0.5) node [red] {{\it unit}};
\path (yz) ++(1.5,2.5) node [red] {{\it multiplication}};
\path (yz) ++(1.5,6) node [red] {{\it symmetry}};
\path (yz) ++(1.5,9) node [red] {{\it comultiplication}};
\path (yz) ++(1.5,10.5) node [red] {{\it counit}};

\draw [dashed,red] (yu) ++ (-1,10) -- +(10,0);
\draw [dashed,red] (yu) ++ (-1,8) -- +(10,0);
\draw [dashed,red] (yu) ++ (-1,1) -- +(10,0);
\draw [dashed,red] (yu) ++ (-1,4) -- +(10,0);

\end{tikzpicture}
}
\caption{\label{rlnstring} A string diagram for a relation}
\end{center}
\end{figure}
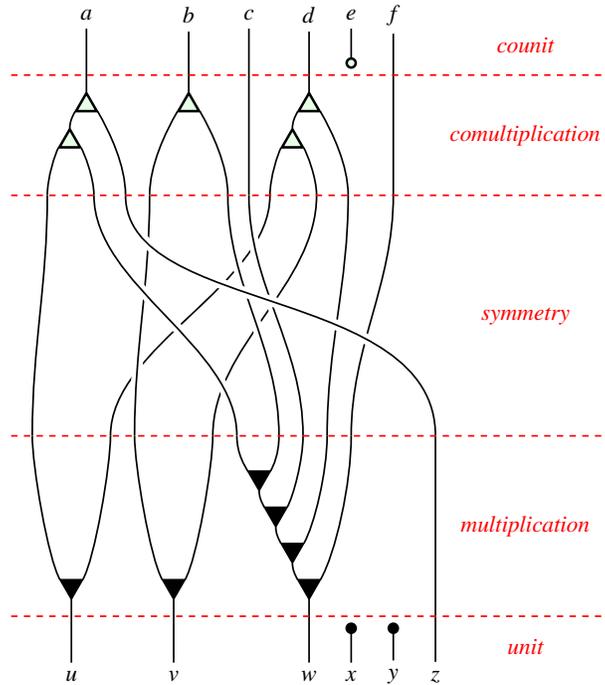

If $A$ is an algebra over $\Mat(\B)$ equippped with a map of rigs $\phi: \B \rightarrow \End(A)$ and $a \in \End(A)$, then $a + a = (\phi(1_R) + \phi(1_R))a = a$; conversely, if $a+a = a$ for all $a \in \End(A)$ then $\phi(1_R) + \phi(1_R) = \phi(1_R)$, so the algebras of $\Mat(\B)$ are exactly those where $a+a = a$ for all $a \in \End(A)$.

Let $A$ be a bimonoid; we say $A$ is {\it special} if $\mu_A \Delta_A = \id_A$. Note that
\[
  a+a = \mu_A(a \otimes a)\Delta_A = \mu_A \Delta_A a,
\]
so it follows that $a+a = a$ for all $a \in \End(A)$ if and only if $A$ is special. Therefore we have

\begin{cor}
The category $\FinRel$ whose objects are finite sets and whose morphisms are relations between them is the PROP for special bicommutative bimonoids.
\end{cor}

Replacing $1_R + 1_R = 1_R$ with $1_R + 1_R = 0_R$ gives us $a+a=0$ for all $a \in \End(A)$, so the above discussion yields:

\begin{cor}
$\Mat(\mathbb{F}_2)$ is the PROP for bicommutative bimonoids $A$ where $\mu_A \Delta_A = \eta_A \epsilon_A$.
\end{cor}

Now consider the case $\Mat(\N)$. This category is equivalent to $\FinSpan$ whose objects are finite sets and whose morphisms $X \rightarrow Y$ are triples $(S, f:S \rightarrow X, g: S \rightarrow Y)$, where $S$ is a set and $f, g$ are ordinary functions. To see this, if $M$ is a $\N$-valued matrix, we can think of $m_{ij}$ as the cardinality of the set
\[
  \{ s \in S | f(s) = j  \text{ and } g(s) = i \}.
\]
Since $\N$ is an inital object in the category of rigs, 

\begin{cor}
$\FinSpan$ is the PROP for bicommutative bimonoids.
\end{cor}

An equivalent version of this result was proved by Lack \cite{Kafo4,Lack}.

For a symmetric monoidal category $\Cc$, define a {\it Hopf monoid} $A$ over $\Cc$ to be a bimonoid equipped with an additional morphism $S: A \rightarrow A$ satisfying
\[
  \mu_A ( S \otimes \id_A ) \Delta_A = \eta_A \epsilon_A = \mu_A (\id_A \otimes S ) \Delta_A.
\]
In the case where $\Cc$ is the category of vector spaces over a field $k$, we call $A$ a {\it Hopf algebra}; thorough explainations of Hopf algebras are given in \cite{Ber,Swe}. We know that $S$ is a bimodule antihomomorphism \cite{Swe}, so if $A$ is bicommutative, $S$ is a bimodule homomorphism. Therefore the enriched structure on $\End(A)$ simplifies the above axiom to
\[
  S + \id_A = 0_A.
\]
Now that we know that $S$ is a bimodule homomorphism, we can say that
\begin{align*}
  0 &= \eta_A \epsilon_A = S \eta_A \epsilon_A = S \mu_A (S \otimes \id_A)\Delta_A 
    = \mu(S^2 \otimes S) \Delta_A = S^2 + S,
\end{align*}
from which it follows that $S^2 = \id_A$.

\begin{cor}
$\Mat(\Z)$ is the PROP for bicommutative Hopf monoids.
\end{cor}
\begin{proof}
We must show that a bicommutative Hopf monoid is exactly a bicommutative bialgebra $A$ equipped with a rig homomorphism $\phi: R \rightarrow \End(A)$.

If $A$ is such an algebra, then $A$ is a bicommutative Hopf monoid with antipode $\phi(-1)$, since
\[
  \eta_A \epsilon_A = \phi(0) = \phi(1) + \phi(-1) = \mu_A (\id_A \otimes \phi(-1))\Delta_A
\]
and
\[
  \eta_A \epsilon_A = \phi(0) = \phi(-1) + \phi(1) = \mu_A (\phi(-1) \otimes 
  \id_A)\Delta_A.
\]

Conversely, Let $A$ be a bicommutative Hopf monoid with antipode $S$. We need to show that $\phi(-1) = S$ extends uniquely to a rig homomorphism $\phi: \Z \rightarrow \End(A)$. We are locked into choosing $\phi(0) = 0$, $\phi(n) = \id_A + \phi(n-1)$, and $\phi(-n) = S + \phi(-n+1)$ inductively for each positive integer $n$. Since $S + \id_A = 0$ is the antipode axiom, we have that $\phi$ is an additive homomorphism of commutative monoids. That $\phi$ is a homomorphism of rigs now follows from the distributive law for rigs and the fact that $S^2 = \id_A$ since $A$ is a bicommutative Hopf monoid.
\end{proof}

\subsection*{Acknowledgements}

We thank John Baez for his help and for his ``Theorems into Coffee" series of blog posts, which posed this paper's central questions \cite{Kafo1,Kafo2,Kafo3,Kafo4}. We also thank Jason Erbele for allowing us to use his diagrams and for his assistance with TikZ.


\bibliographystyle{plain}

\end{document}